\documentclass[onefignum,onetabnum]{siamart251216}

\usepackage{lipsum}
\usepackage{amsfonts}
\usepackage{graphicx}
\usepackage{epstopdf}
\usepackage{algorithmic}

\def\N{\mathbb N}

\def\R{\mathbb R}

\def\R{\mathbb{R}}
\def\N{\mathbb{N}}

\def\cB{\mathcal{B}}

\def\cI{\mathcal{I}}

\def\cO{\mathcal{O}}

\def\aa{\boldsymbol{a}}

\def\ee{\boldsymbol{e}}

\def\rr{\boldsymbol{r}}

\def\tt{\boldsymbol{t}}

\def\xx{\boldsymbol{x}}
\def\yy{\boldsymbol{y}}

\def\11{\mathbf{1}}

\def\00{\boldsymbol{0}}

\def\aalpha{\boldsymbol{\alpha}}

\def\oomega{\boldsymbol{\omega}}

\def\fraku{\mathfrak{u}}

\def\transpose{\text{T}}

\ifpdf
  \DeclareGraphicsExtensions{.eps,.pdf,.png,.jpg}
\else
  \DeclareGraphicsExtensions{.eps}
\fi

\newsiamremark{remark}{Remark}
\newsiamremark{hypothesis}{Hypothesis}
\crefname{hypothesis}{Hypothesis}{Hypotheses}
\newsiamthm{claim}{Claim}

\headers{Nodal Representations for Kernel-Based Multilevel
  Interpolation}{L. Gollwitzer, R. Kempf, and H. Wendland} 

\title{Nodal Representations for Kernel-Based Multilevel Interpolation\thanks{Submitted to the editors \today.
\funding{The work of L. Gollwitzer and H. Wendland was funded by the Deutsche Forschungsgemeinschaft (DFG, Germand Research Foundation) - Project Number 452806809.}}}

\author{Lorenz Gollwitzer \thanks{University of Bayreuth, Germany 
  (\email{lorenz.gollwitzer@uni-bayreuth.de}).}
\and R\"udiger Kempf \thanks{University of Bayreuth, Germany
  (\email{ruediger.kempf@uni-bayreuth.de}).}
\and Holger Wendland  \thanks{University of Bayreuth, Germany 
  (\email{holger.wendland@uni-bayreuth.de}).}
}
\usepackage{amsopn}

\makeatletter
\newcommand*{\addFileDependency}[1]{
  \typeout{(#1)}
  \@addtofilelist{#1}
  \IfFileExists{#1}{}{\typeout{No file #1.}}
}
\makeatother

\ifpdf
\hypersetup{
  pdftitle={Nodal Representations for Kernel-Based Multilevel Interpolation},
  pdfauthor={L. Gollwitzer, R. Kempf, and H. Wendland}
}
\fi

\usepackage{xcolor}

\begin{document}

\maketitle

\begin{abstract}
We study the kernel-based multilevel method for approximating or
learning a multivariate function from scattered data, motivated in
part by recent applications in sparse 
grid methods. For nested families of point sets, we derive a
 nodal representation of the multilevel interpolant that
depends explicitly on the values of the target function. This representation
allows us to characterize the range of the associated interpolation
operator and to construct a cardinal basis for this space. We prove
that the resulting basis functions exhibit exponential decay,
analogous to the localization properties known for certain kernel-based
Lagrange functions. We further analyze the computational cost of
the resulting formulation. For non-nested families of point sets,
we derive a generalized nodal representation.
\end{abstract}



\section{Introduction}
\label{kempf:sec:1}

Kernel-based methods have become a well-established tool in scattered
data approximation \cite{Buhmann-03-1,Fasshauer-07-1,Wendland-05-1},
meshless numerical methods for partial differential equations
\cite{Fornberg-Flyer-15-1,Franke-Schaback-98-1,Wendland-99-2} and machine learning
\cite{Cristianini-Shawe-Taylor-00-1,Cucker-Zhou-07-1,Schoelkopf-Smola-02-1,
  Steinwart-Christmann-08-1}. For
a  broader overview see also \cite{Schaback-Wendland-06-1}. 
They are particularly valued for their flexibility, stability and
deterministic error analysis
\cite{Narcowich-etal-05-1,Narcowich-etal-06-1,Schaback-95-2}.   

A major difficulty arises when the number of data becomes large:
the kernel interpolation matrix may become severely ill-conditioned
\cite{Schaback-95-2}. A common remedy is to introduce a scaling
parameter and decrease the correlation length if the number of
points grows. If chosen appropriately, this leads to
well-conditioned interpolation matrices. However, then,
convergence of the approximation method is lost. This phenomenon is
known as the \emph{trade-off principle}: either the approximation
method converges but the system matrix becomes ill-conditioned, or the
system matrix is well-conditioned but convergence is no longer
guaranteed \cite{Schaback-95-2, Schaback-97-3}.  

For compactly supported kernels, ideas have been developed to overcome
this trade-off principle by representing multiple scales within the
data \cite{Floater-Iske-96-1,Schaback-97-3}. The algorithm follows a residual
correction strategy. First, an approximation is computed on a coarse
set of data sites using a large kernel support, capturing large-scale
features of the target function. The residual between the target
function and this approximation typically contains finer-scale
information. This residual is then approximated on a finer set of
points using a smaller kernel support. The procedure is iterated
until the desired accuracy is reached. In the interpolation setting,
the sum of the levelwise interpolants is again an interpolant of the
target function on the finest set of points. Although the
procedure was established early and was shown to converge numerically,
theoretical results were established much later in
\cite{LeGia-etal-10-1,Wendland-10-1}. This method was also studied in
\cite{Avesani-etal-25-1, Hales-Levesley-02-1,Narcowich-etal-99-1, Usta-Levesley-18-1} for
globally supported kernels and was used for solving partial
differential equations in, for example, \cite{Chen-etal-02-1,
  LeGia-etal-12-1,Fasshauer-99-1, Wendland-18-1}.

Recently, this multilevel approach has been combined with Smolyak's
construction to obtain kernel methods for high-dimensional function
reconstruction \cite{Buettner-etal-25-1, Georgoulis-etal-13-1, Griebel-etal-26-1, Kempf-Wendland-23-1}. In this setting the
multilevel operator is tensorized, which requires an explicit
representation of its dependence on the data. However, the standard
formulation expresses the operator implicitly through recursively
defined residuals. Although first results in this direction were
obtained in the aforementioned works, the resulting formulations are
computationally infeasible in practice. 
In this paper we derive a nodal representation of the multilevel
approximation operator that makes its dependence on the data
explicit. This representation yields new structural insights into the
global approximation space. In particular, for nested families of
point sets we construct both a Lagrange basis and a Newton-type
basis of the multiscale space. These bases allow us to characterize
the range of the multilevel operator and to analyze localization
properties and computational cost.

The analysis is more involved for non-nested sets of data
sites. Nevertheless, the ideas developed for the nested case can still
be used to construct a basis of the global approximation space and to
characterize the range of the multilevel operator. In general, this
basis is no longer cardinal. Such non-nested configurations arise
naturally in adaptive variants of multilevel methods. 

These results provide a new structural understanding of multilevel
kernel methods and lead to practical representations of the
approximation operator that are suitable for high-dimensional
constructions. 

The remainder of the paper is organized as
follows. \Cref{sec:MultilevelMethod} introduces notation and reviews
the multilevel approximation method. We also extend the convergence
result of \cite{Wendland-10-1} to a broader range of Sobolev
norms. \Cref{sec:NestedSites} studies nested families of data sites
and derives Lagrange and Newton-type bases for the global
approximation space. Exponential decay of the Lagrange basis functions
is established in
\Cref{subsec:ExponentialDecay}. \Cref{sec:NonNestedSites} treats the
non-nested case. \Cref{sec:Conclusion} provides a
conclusion and outlook.


\section{The Multilevel Method}\label{sec:MultilevelMethod}

The setup of the kernel-based multilevel method is 
usually the following one, cf. \cite{Wendland-10-1}. For
a bounded domain $ \Omega \subseteq \R^d $ we assume 
that we are given a sequence of finite point sets $ X_1, X_2, 
\ldots \subseteq \Omega $, where 
$
 X_\ell = \left\{ \xx_1^{(\ell)}, \dots, \xx^{(\ell)}_{N_\ell} \right\}.
$

Associated to the set $ X_\ell $, $\ell\in\N$, we define the 
\emph{fill distance} $h_\ell$ and {\em separation radius} $q_\ell$ as usual
by
\[
 h_\ell := h_{X_\ell,\Omega} = \sup_{\xx \in \Omega}
 \min_{1\le i\le N_\ell} \| \xx - \xx^{(\ell)}_{i} \|_2, \qquad
 q_\ell := q_{X_\ell} = \frac{1}{2}
 \min_{i \neq j}\| \xx^{(\ell)}_i - \xx^{(\ell)}_j \|_2.
\]
For some results it is important that the data sites are
\emph{quasi-uniform}, that is there exists a  
level-independent constant $ c_{qu}>0 $ such that 
\[
 q_\ell \leq h_\ell \leq c_{qu} q_\ell, \qquad \ell \in \N.
\]
The sets $ X_\ell$ are not required to be nested, however,
we assume that the fill distances  
decay uniformly, i.e., that there is a uniform \emph{refinement 
parameter} $ 0 < \mu < 1 $ and a constant $ c \in (0,1] $,
such that
\[
 c \mu h_\ell \leq h_{\ell+1} \leq \mu h_\ell, \qquad \ell \in \N.  
\]

The second important ingredient for defining our method is a compactly
supported, positive definite radial basis function (RBF) $\Phi: \R^d \to \R$ with
support in the closed unit ball.  Using  \emph{scaling parameter} $
\delta_\ell := \nu h_\ell $ with an $\ell$-independent $\nu>1$, we 
obtain a basis function for every level $ \ell \in \N $, given as
\begin{equation}\label{eq:RescaledKernel}
 \Phi_\ell := \Phi_{\delta_\ell} := \delta^{-d} \Phi(\cdot  / \delta_\ell), \quad \ell \in \N.
\end{equation} 

Noting that $\Phi_\ell$ has support in the closed ball about zero with
radius $\delta_\ell$, we have the \emph{local approximation spaces}
\[
 W_\ell := 
\operatorname{span} \{ \Phi_\ell
(\cdot - \xx^{(\ell)}_{i}) \; : \;  1\le i\le N_\ell \},
\]
and, for $L\in N$,  \emph{global approximation spaces}
\begin{equation}\label{eq:DefApproximationSpace}
V_L=W_1+\cdots +W_L.
\end{equation}
Under the above assumption, it can be shown that 
the sum in \eqref{eq:DefApproximationSpace} is direct, see
\cite{LeGia-etal-10-1}. Moreover, we can give an alternative basis for
these spaces. To be more
precise for every level $\ell$ and every $1 \le i \le N_\ell$ there
are functions $\chi_i^{(\ell)}\in W_\ell$ satisfying
$\chi_i^{(\ell)}(\xx_j^{(\ell)})=\delta_{ij}$. They are given  by
$\left(\chi^{(\ell)}_{1}, \ldots, \chi^{(\ell)}_{N_\ell}
 \right) = \rr_\ell^{\transpose} M_\ell^{-1},
$
where $ \rr_\ell =  (\Phi_\ell(\cdot - \xx_{i}^{(\ell)}))_i^{\transpose}: \Omega\to
\R^{N_\ell} $ and  $ M_\ell = (\Phi_\ell(\xx_{i}^{(\ell)} -
\xx_{j}^{(\ell)})_{ij})\in\R^{N_\ell\times N_\ell}$ is the positive definite
kernel matrix.

\begin{definition}
A  basis of the global approximation space  $V_L$ is defined  by
\[
\cB_{\operatorname{full}}:=\{\chi_i^{(\ell)} : 1\le \ell \le L, 1\le i\le
N_\ell\}.
\]
\end{definition}
Obviously, this is indeed a basis of $V_L$, showing that the dimension
of the global approximation space $V_L$ is $N_1+\cdots + N_L$. In the
case of nested data sets $X_\ell\subseteq X_{\ell+1}$ the basis
$\cB_{\operatorname{full}}$ is  a Newton-type basis,
as we have for any $1\le \ell\le L$ that $\chi_i^{(\ell)}(\xx_j^{(m)})=0$ for
all points from $X_m$, $1\le m\le \ell$ but not for the later added points from
$X_{\ell+1}$ to $X_L$, i.e. it satisfies the Newton property of a basis
{\em batch-wise}.

It is well-known that the local approximation
spaces $W_\ell$ are not rich enough and that the global spaces $V_L$ are
redundant. Hence, it is important to properly choose a
good subspace of $V_L$, which is still rich enough but less
redundant. We will achieve this by taking the image of a multilevel
method, which are describing now.

The \emph{ kernel-based multilevel method} depicts one way of successively
choosing local approximants $s_\ell\in W_\ell$ to form a global approximant
$f_\ell=s_1+\cdots+s_\ell\in V_\ell$. To be more precise, the method starts
with an initial approximation $f_0=0$ and an initial error $e_0=f$ and
then proceeds for $\ell\in\N$ by computing a local approximation
$s_\ell\in W_\ell$ to the error $e_{\ell-1}$. Then, the
global approximation and the error are updated via
\begin{align*}
  f_\ell & = f_{\ell-1}+s_\ell,\\
  e_\ell &= e_{\ell-1}-s_\ell.
\end{align*}
Numerically, the procedure stops after, say, $L$ levels, but for
theoretical results we may assume that $L$ tends to infinity.

In this paper, we are particularly interested in interpolation as the
means of determining the local approximation $s_\ell\in W_\ell$. Thus, we
set $s_\ell=I_{X_\ell,\Phi_\ell} (e_{\ell-1})$ for any $\ell \in \N $, where the
\emph{interpolation operator} $ I_{X_\ell, \Phi_\ell}: C(\Omega) \to W_\ell $
 is defined by 
\[
  I_{X_\ell,\Phi_\ell}(g) = \sum_{i=1}^{N_\ell} g(\xx_i^{(\ell)}) \chi_i^{(\ell)}, \qquad g \in C(\Omega).
\]
 
 \begin{definition}\label{def:MultilevelOperator}
With the notation and assumptions made above, the
\emph{(interpolatory) multilevel operator} $ A_L: C(\Omega) \to  V_L $
is defined by
\[
  A_L(f) := \sum_{\ell=1}^{L} I_{X_\ell,\Phi_\ell}(e_{\ell-1}).
  \]
  Its image will be denoted by $\widetilde{V}_L:=A_L(C(\Omega))$. 
\end{definition}

The following simple observation will be crucial in our analysis. 
\begin{lemma}
The operator $A_L$  is linear and, for any $f\in C(\Omega)$, $A_L f$
interpolates $ f $ on $ X_L $,  i.e. $ A_L(f)(\xx) = f(\xx) $ for $
\xx \in X_L $. 
\end{lemma}

\begin{proof}
 Linearity of the operator follows from an easy induction
 on the number of levels and the observation that the 
 concatenation of linear operators is linear. The proof 
 of the interpolation 
 property can be found in, e.g., \cite{Wendland-05-1}. 
\end{proof}

We end this section by extending the error result of \cite[Theorem 1]{Wendland-10-1}
to more general Sobolev norms. While this is obviously a
straight-forward generalization, it seems that it has not been stated
before and is important for itself. It also allows us to specify the
choice of the RBF $\Phi$ in more details. 

A Hilbert space $ H$ of functions $ f: \Omega \to \R $ is called a
\emph{reproducing kernel Hilbert space (RKHS)} if there is a kernel $
K: \Omega \times \Omega \to \R $ such that $ K(\cdot, \xx) \in H $ and
$ f(\xx) = \langle f, K(\cdot, \xx) \rangle_H $ for all $ \xx \in
\Omega $. $ K $ is then called the \emph{reproducing kernel} of $ H
$. If $\Omega\subseteq\R^d$ has a Lipschitz boundary or if
$\Omega=\R^d$ then the Sobolev space $H^\sigma(\Omega)$ with
$\sigma>d/2$ is a RKHS. If the Fourier transform of the RBF $ \Phi: \R^d \to \R $
satisfies  
\begin{equation}\label{eq:AlgebraicDecay}
 c_1(1+ \| \oomega \|_2^2)^{-\sigma} \leq \widehat{\Phi}
 (\oomega) \leq c_2(1+ \| \oomega \|_2^2)^{-\sigma}, \qquad \oomega \in \R^d,
\end{equation}
with constants $ c_1, c_2 > 0 $, then $ K: \R^d \times \R^d \to \R $
defined by $ K(\xx,\yy) := \Phi(\xx - \yy) $ is the reproducing kernel
of $ H^{\sigma}(\R^d) $ if equipped with the norm  
\begin{equation}\label{eq:NativeSpaceNorm}
 \| f \|_{\Phi} := \int_{\R^d} \frac{ |
   \widehat{f}(\oomega)|^2}{\widehat{\Phi}(\oomega)} \ d \oomega. 
\end{equation}
It is easy to see that this norm is equivalent to the  usual norm on
$H^{\sigma}(\R^d) $. Furthermore, a simple Fourier argument shows that
also our scaled RBFs $\Phi_\ell$ are reproducing kernels of
$H^\sigma(\R^d)$. However, then, the equivalence constants depend on $
\delta_\ell $, i.e. we have $\widetilde{c}_1 \|f\|_{\Phi_{\delta_\ell}} \le
\|f\|_{H^\sigma(\R^d)}\le \widetilde{c}_2 \delta_\ell^{-\sigma}
\|f\|_{\Phi_{\delta_\ell}}$ for all $f\in H^\sigma(\R^d)$. For more
details, we refer to \cite{Wendland-10-1}. 

The next theorem is a sampling inequality, taken from
\cite{Arcangeli-etal-12-1}, which is a slight generalization of the
sampling inequality originally derived in \cite{Narcowich-etal-05-1}.
%

\begin{theorem}\label{thrm:GeneralSamplingInequality}
 Let $ \Omega \subseteq \R^d $ be a bounded Lipschitz domain.
 Let $ q \in [1,\infty] $ and $ \sigma > d/2 $.
 Then there exist constants $ h_0 > 0 $ and $ C > 0 $ such 
 that for all $ X = \{\xx_1, \dots, \xx_N \} \subseteq \Omega $
 with $ h = h_{X, \Omega} \leq h_0 $ and all $ u \in H^{\sigma}
 (\Omega) $ with $ u|_X = 0 $, the bound 
\[
  \| u \|_{W_q^{\tau}(\Omega)} \leq C h^{\sigma - \tau
  - d(\frac{1}{2} - \frac{1}{q})_+} \| u \|_{H^{\sigma}(\Omega)}
\]
 holds for all $ 0 \leq \tau \leq \widetilde{\sigma} $, where $ \widetilde{\sigma} = \widetilde{\sigma}_0
 := \sigma - d (1/2 - 1/q)_+ $ in the case $ \sigma \in \N $ and 
 either $ q > 2 $ and $ \widetilde{\sigma}_0 \in \N $ or $ q = 2 $. Otherwise,
 $ \widetilde{\sigma} = \lceil \widetilde{\sigma}_0 \rceil - 1
 $. Moreover, $ \tau \in \N_0 $ in   the case $ q = \infty $.
\end{theorem}

Now we are in the position to state the easy extension of the convergence result for 
the multilevel method. As in the original result in \cite{Wendland-10-1},
which only gives $L_2$-error bounds, the sequence of data sites $X_\ell$
is not required to be quasi-uniform.

\begin{theorem}\label{thrm:ErrorEstimate}
Let $ \Omega \subseteq \R^d $ be a bounded Lipschitz domain. 
Let $ X_1, X_2, \ldots \subseteq\Omega $ be a sequence of point 
sets in  $ \Omega $ with fill distances $ h_1, h_2, \dots $ satisfying
$ c \mu h_\ell \leq h_{\ell+1} \leq \mu h_\ell $ for $ \ell\in\N$
with fixed $ \mu \in (0,1) $, $ c \in (0,1] $ and $ h_1 $ 
sufficiently small. Let $ \sigma > d/2 $ and $ \Phi $ be the 
reproducing kernel of 
$ H^{\sigma}(\R^d) $, i.e., its Fourier transform satisfies 
\eqref{eq:AlgebraicDecay}, and let $ \Phi_\ell $ be defined by 
\eqref{eq:RescaledKernel} with scale factor $ \delta_\ell = 
\nu h_\ell $. Assume $ 1 /h_1 \geq \nu \geq \gamma / \mu $ with  
a fixed $ \gamma > 0 $.  Let $ q \in [1, \infty] $. Then there exist
constants $ C,C_1 > 0 $ such that for all $ f \in H^{\sigma}(\Omega)
$, with $ \alpha = C_1 \mu^{\sigma    - \tau - d (\frac{1}{2} - \frac{1}{q})_+} $,
 \begin{equation}\label{eq:ErrorEstimate}
  \| f - A_L(f) \|_{W^{\tau}_q(\Omega)} \leq C \alpha^{L} 
  \| f \|_{H^{\sigma}(\Omega)},
  \end{equation}
 holds for all $ 0 \leq \tau \leq \widetilde{\sigma} $, where $
 \widetilde{\sigma} = \widetilde{\sigma}_0 
 := \sigma - d (1/2 - 1/q)_+ $ in the case $ \sigma \in \N $ and 
 either $ q > 2 $ and $ \widetilde{\sigma}_0 \in \N $ or $ q = 2 $. Otherwise,
 $ \widetilde{\sigma} = \lceil \widetilde{\sigma}_0 \rceil - 1 $. Moreover, $ \tau \in \N_0 $ in
 the case $ q = \infty $.
 \end{theorem}

\begin{proof}
  We employ a specific recursion derived in  \cite[Proof of Theorem
    1]{Wendland-10-1}. If $ E: H^{\sigma}(\Omega) \to
  H^{\sigma}(\R^d) $ denotes Stein's  extension operator for Sobolev functions
  (see  \cite{Stein-71-1}) then this recursion is given by
 \begin{equation}\label{eq:RecursionEstimate}
  \| E e_\ell \|_{\Phi_{\ell+1}} \leq C_1 \mu^{\sigma}
  \| E e_{\ell-1} \|_{\Phi_\ell},
 \end{equation}
 where we use the norm of \eqref{eq:NativeSpaceNorm} induced by $
 \Phi_{\ell+1} $ and $ \Phi_\ell $, respectively.  

 The error on level $ L $ satisfies $ e_L |_{X_L} = 0 $. This allows
 us to use the sampling inequality from
 \cref{thrm:GeneralSamplingInequality} to obtain, in complete analogy
 to \cite[Theorem 1]{Wendland-10-1}, the bound 
 \begin{align*}
  \| e_L \|_{W^{\tau}_q(\Omega)} &\leq C
  h_{L}^{\sigma - \tau
  - d (\frac{1}{2} - \frac{1}{q})_+} \| e_L \|_{H^{\sigma}
  (\Omega)} \\
  &\leq C h_{L}^{\sigma - \tau
  - d (\frac{1}{2} - \frac{1}{q})_+} \| Ee_L \|_{H^{\sigma}
  (\R^d)} \\
  &\leq C h_{L}^{\sigma - \tau
  - d (\frac{1}{2} - \frac{1}{q})_+} \delta^{-\sigma}_{L+1} 
  \| Ee_L \|_{\Phi_{L+1}} \\
  & = C \nu^{-L} h_L^{- \tau
  - d (\frac{1}{2} - \frac{1}{q})_+} \| Ee_L \|_{\Phi_{L+1}} \\
  &\leq C C_1^L (\mu^L)^{- \tau
  - d (\frac{1}{2} - \frac{1}{q})_+} (\mu^L)^{\sigma}
  \| Ef \|_{\Phi_{1}} \\
  &= C \left(C_1 \mu^{\sigma - \tau
  - d (\frac{1}{2} - \frac{1}{q})_+} \right)^L \| Ef \|_{\Phi_1},
 \end{align*}
 where we used the recursion  \eqref{eq:RecursionEstimate}
 $ L $-times. 
 The stated estimate then follows from 
$
  \| Ef \|_{\Phi_1} \leq C \| Ef \|_{H^{\sigma}(\R^d)} 
  \leq C \| f \|_{H^{\sigma}(\Omega)},
$
 using the norm equivalence and the properties of the extension operator $ E $.
\end{proof}

In particular, these error estimates imply that the 
multilevel operator $ A_L $ is bounded.

\begin{corollary}
 With the notation and under the assumptions of
 \cref{thrm:ErrorEstimate}, the multilevel operator $A_L$ is bounded
 with norm satisfying
\[
  \|A_L\|_{H^{\sigma}(\Omega) \to W^{\tau}_q(\Omega)} \leq 
  C \left(\alpha^L + 1 \right).
\]
\end{corollary}

\begin{proof}
We see that 
$
 \| A_L(f) \|_{ W^{\tau}_q(\Omega)} \leq 
 \| A_L(f) - f \|_{ W^{\tau}_q(\Omega)} + \|
 f \|_{ W^{\tau}_q(\Omega)}.
$
 Using the error result \eqref{eq:ErrorEstimate} and the Sobolev
 embedding theorem in the form
 $ \| f \|_{W^{\tau}_q(\Omega)} \leq C \| f \|_{H^{\sigma}
 (\Omega)} $ completes the proof.
\end{proof}

\section{The Multilevel Method in Nodal Representation For 
  Nested Data Sets}\label{sec:NestedSites}

Throughout this section, we will assume that the data sets $X_\ell$ are
nested, i.e. that we have $X_{\ell}\subseteq X_{\ell+1}$. We will also
  assume that the points are ordered in such a way that, for $1\le \ell\le
  L$, the first  $N_\ell$ points of $X_L$ belong to $X_\ell$, i.e.
  \[
  X_{\ell} = \{\xx_1,\ldots,\xx_{N_\ell}\},
  \]
  allowing us to drop upper indices to indicate the current
  level. This also means that $X_\ell\setminus X_{\ell-1} =
  \{\xx_{N_{\ell-1}+1},\ldots, \xx_{N_\ell}\}$ are those points which are
  newly introduced on level $\ell$. Moreover, if we set $N_0=0$ then for
  every $1\le i\le N_L$ 
  there is an $1\le \ell \le L$ such that $N_{\ell-1}+1 \le i \le N_{\ell}$,
  meaning that the point $\xx_i$ belongs to $X_{\ell}\setminus X_{\ell-1}$,
  i.e. it is first introduced into the data set at level $\ell$.

In this situation, we now want to derive new ways of representing the
multilevel operator $A_L$, i.e. we are looking for good bases of
$\widetilde{V}_L=A_L(C(\Omega))$.  First ideas in this direction have been introduced in 
\cite[Section 3.3]{Kempf-Wendland-23-1} and
\cite{Buettner-etal-25-1}, which we will refine here.  

\subsection{Lagrange and Newton-type Bases of the Approximation
  Spaces}

Though we have introduced the multilevel operator $A_L$ as an operator
$A_L:C(\Omega)\to V_L$, it actually uses only $f|X_L$ to compute
$A_L(f)$ for any $f\in C(\Omega)$ if the data sets $X_\ell$ are
nested. For every vector $\yy\in\R^{N_L}$ we can obviously find a
function $f\in C(\Omega)$ with $f|X_L=\yy$,  which means that we can,
with a slight abuse of notation, also consider the multilevel operator
as a mapping $A_L:\R^{N_L}\to V_L$. 
This leads to the following first
definition of a basis for $\widetilde{V}_L=A_L(C(\Omega))= A_L(\R^{N_L})$.

\begin{definition}
Denote the $i$-th unit vector in $\R^{N_L}$ by $\ee_i^{(L)}$. Then,
the \emph{multilevel cardinal basis} $\cB=\{b_1,\ldots, b_{N_L}\}$ for
the approximation space $\widetilde{V}_L$ is defined by
\[
b_i = A_L\left(\ee_i^{(L)}\right), \qquad 1\le i\le N_L.
\]
\end{definition}
Note that we also have $b_i=A_L(\chi_i^{(L)})$, using the
previously introduced Lagrange functions, as we have $\chi_i^{(L)}|X_L=\ee_i^{(L)}$.
Writing $f|X_L$ as
$f|X_L = \sum_{i=1}^{N_L} f(\xx_i) \ee_i^{(L)}$ and the linearity of
$A_L$ immediately yield
\[
A_L f = \sum_{j=1}^{N_L} f(\xx_i) A_L\left(\ee_i^{(L)}\right) = \sum_{j=1}^{N_L}
f(\xx_i) b_i,
\]
showing that $\cB$  spans $\widetilde{V}_L$. It is also indeed a
cardinal or Lagrange basis.

\begin{theorem}\label{thm:basis1}
The multilevel cardinal basis  is indeed a basis of $\widetilde{V}_L =
A_L(C(\Omega))$, showing particularly $\dim \widetilde{V}_L=N_L$. Moreover, the
basis functions are cardinal in the sense that 
$b_i(\xx_j) = \delta_{ij}$ for $1\le i,j\le N_L$.
\end{theorem}
\begin{proof}
The cardinal condition follows from the fact that $A_Lf$
interpolates $f$ on $X_L$. Here, this means
$
b_i(\xx_j) = A_L(\chi_i^{(L)})(\xx_j) = \chi_i^{(L)}(\xx_j) = \delta_{ij}.
$
This condition also guarantees that the functions $b_i$ are linearly
independent. This means  $\dim \operatorname{span}(\cB) = N_L$
and the above considerations have shown $\widetilde{V}_L\subseteq
\operatorname{span}(\cB)$.  Finally, any $f\in
\operatorname{span}(\cB)$ can be written as $f=\sum \alpha_i
b_i$. This, however, means $f=A_L(\aalpha)$ and thus $f\in
\widetilde{V}_L$ showing $\widetilde{V}_L = \operatorname{span}(\cB)$.
\end{proof}

The above result shows that $\widetilde{V}_L$ is indeed a proper
subspace of $V_L$, as the latter has dimension $N_1+\cdots + N_L$,
while the former has only dimension $N_L$.

To better understand this multilevel cardinal basis, we will now
derive several other representations. To this end, we will use 
an alternative way of expressing  the multilevel operator, see 
\cite{Buettner-etal-25-1,Kempf-Wendland-23-1}. In what follows,
 we will always assume that subsets $\fraku = \{u_1,\ldots, 
u_{|u|}\} \subseteq \{1,\ldots, L\}$ are ordered as $ u_1 < u_2 <
\cdots < u_{|\fraku|}$.

\begin{theorem}\label{thrm:OldNewRepresentation}
 The   multilevel operator $ A_L $ can be represented as
\[
  A_L = \sum_{\substack{\fraku \subseteq \{1, \ldots, L \}
  \\ 1 \leq |\fraku| \leq L}} (-1)^{|\fraku|+1} \cI_{\fraku}
  = \sum_{j = 1}^{L} (-1)^{j+1} \sum_{\substack{\fraku 
  \subseteq 
  \{1, \dots, L \}\\ |\fraku| = j}} \cI_{\fraku},
\]
 where $\cI_{\fraku}$ is the identity if $\fraku=\emptyset$ or
 otherwise $
 \cI_{\fraku}:C(\Omega)\to W_{|\fraku|} $ is the combined
 operator  
 \begin{equation}\label{eq:CombinedInterpolationOperator}
  \cI_{\fraku} = I_{X_{\fraku_{|\fraku|}}, \Phi_{\fraku_{|\fraku|}}}
	I_{X_{\fraku_{|\fraku|-1}}, \Phi_{\fraku_{|\fraku|-1}}} \cdots 
	I_{X_{u_1}, \Phi_{u_1}}.
  \end{equation}
\end{theorem}
The ordering in (\ref{eq:CombinedInterpolationOperator}) is
important, as the interpolation operators do not commute.

To apply this representation, we note that for any index $N_{m-1}
+1 \le i \le N_m$ with $1\le m\le L$ and any $1\le u_1\le 
m-1$ we have $\chi_i^{(L)}|X_{u_1}=0$,
showing particularly
$
I_{X_{u_1},\Phi_{u_1}} (\chi_i^{(L)}) = 0.
$
This immediately yields the following result, which is essentially a reformulation of 
\cite[Theorem 3.9]{Buettner-etal-25-1} and provides 
a first alternative representation of $\cB=\{b_i : 1 \leq i \leq N_L\}$.

\begin{corollary}\label{thrm:ReprMultilevelOperator}
Assume that the index $i$ satisfies $N_{m-1}+1\le i\le N_m$ for
some $1\le m \le L$. Then, the corresponding basis function has the
representation  
\begin{equation}\label{eq:DefBi}
 b_i = \sum_{\emptyset \neq \fraku \subseteq
 \{m, \dots, L\}} (-1)^{|\fraku| +1}
 \cI_{\fraku} \left( \chi^{(L)}_{i} \right), 
 \quad 1 \leq i \leq N_L,
\end{equation}
with $ \cI_{\fraku} $ as in 
\eqref{eq:CombinedInterpolationOperator}.
\end{corollary}
\begin{proof}
  This immediately follows from
  \[
  b_i = A_L\left(\ee_i^{(L)}\right)  = \sum_{\substack{\fraku \subseteq \{1, \ldots, L \}
  \\ 1 \leq |\fraku| \leq L}} (-1)^{|\fraku|+1}
  \cI_{\fraku}\left(\ee_i^{(L)}\right)
  \]
  and the fact that  for all $\fraku$ with $u_1\le m-1$ we have
  $\cI_{\fraku}(\chi_i^{(L)}) = 0$. 
\end{proof}

While this is already an improvement, it still sums over too many
sets $\fraku$.  As before,  we can alternatively use
$\cI_{\fraku}(\ee_i^{(L)})$ instead of $\cI_{\fraku}(\chi_i^{(L)})$ in
\eqref{eq:DefBi}.

For a more expressive representation of the function
$ b_i $, we need the following lemma concerning 
interpolation of Lagrange functions. The idea of its proof has already
been laid out above.

\begin{lemma}\label{lem:InterpolationLagrangeFunc}
 Let $ X_\ell \subseteq X_L \subseteq \Omega $ be two
 sets of data sites with associated cardinal functions  $
 \{\chi^{(\ell)}_{i} : 1 \leq i \leq N_\ell\}  
 \subseteq W_\ell $ and $ \{\chi^{(L)}_{i} : 1 \leq i \leq N_L\} 
 \subseteq W_L  $, respectively. Then, for any $1\le i\le N_L$ we have 
\[
   I_{X_\ell, \Phi_\ell} \left(\chi^{(L)}_{i} \right) = 
   \begin{cases}
    \chi^{(\ell)}_{i},& \quad 1
                   \leq i \leq N_\ell, \\
    0,& \quad N_{\ell}+1 \leq i \leq N_L.
   \end{cases}
\]
\end{lemma}

\begin{proof}
The Lagrange property of $ \chi^{(L)}_{i} $ implies that 
$ \chi^{(L)}_{i}|{X_\ell} = \ee_i \in \R^{N_\ell} $ if $ i \leq 
N_\ell $, i.e. if $ \xx_{i} \in X_\ell $ and $ \chi^{(L)}_{i}|{X_\ell} = \00 \in \R^{N_\ell} $ 
 if $ i > N_\ell $, i.e.  if $ \xx_{i} \notin X_\ell $.  The claim
 then follows from the uniqueness of the interpolant.
\end{proof}

In essence, \cref{lem:InterpolationLagrangeFunc} 
states that interpolating a fine-level Lagrange 
function on a coarse level yields the corresponding
coarse-level Lagrange function if its anchor point lies
in the coarse set of data sites. Otherwise, the result is zero.

We need another idea, which is very similar to the one that we used a
few times before. If $f|X_\ell = 0 $ for some $1\le \ell\le L$ then
the first $\ell$ steps of the multilevel method produce the zero
function as the approximation and the actual approximation starts with
step $\ell+1$. We will use this and hence introduce the notation
$A_{\{\ell+1,\dots,L\}}:C(\Omega) \to V_L$ for the multilevel operator
starting at level $\ell+1$. It is easy to see that the
representation from Theorem \ref{thrm:OldNewRepresentation} becomes
in this situation
\begin{equation}\label{restrictedmultilevel}
A_{\{\ell+1, \dots,L\}}(f) := \sum_{\emptyset \neq
  \fraku \subseteq \{\ell +1, \dots, L\}} (-1)^
{|\fraku|+1} \cI_{\fraku} (f)
\end{equation}

With this we can prove another representation of 
the multilevel cardinal bases $\cB= \{b_i : 1 \leq i \leq N_L\}$.

\begin{theorem}\label{thrm:Bi}
 With the assumptions and notation from above, let $ 1 \leq m
 \leq L $ and $ N_{m-1}+1 \leq i \leq N_m $. Then 
 the function $ b_i $ has the representation
 \begin{equation}\label{eq:NewRepresentationBi}
  b_i = \chi^{(L)}_{i} + \sum_{\ell = m}
  ^{L-1} \left[\chi^{(\ell)}_{i} - A_{\{\ell+1, \dots,L\}} 
  \left(\chi^{(\ell)}_{i} \right) \right].
 \end{equation}
\end{theorem} 

\begin{proof}
We start with the representation of $b_i$ from \eqref{eq:DefBi} in the
form
\[
 b_i = \sum_{\emptyset \neq \fraku \subseteq
 \{m, \dots, L\}} (-1)^{|\fraku| +1}
 \cI_{\fraku} \left( \chi^{(L)}_{i} \right).
\]
Next, we order the sets $ \emptyset \neq \fraku \subseteq
 \{m, \dots, L \} $ according to their 
 smallest element, which we denote by $ \ell$ and which obviously
 satisfies $\ell \ge m $. Thus, we set  $ 
 \fraku = \{\ell \} \cup \widetilde{\fraku} $, where  
 $ \widetilde{\fraku} \subseteq \{\ell +1, \dots, L \} $.
 In particular, we have for $ \ell = L $ that $ \fraku = 
 \{L\} $, $ \widetilde
 {\fraku} = \emptyset $ and 
\[
  \sum_{\fraku = \{L\}} (-1)^{| \fraku| +1} \cI_{\fraku}
  (\chi^{(L)}_{i}) = 
  I_{X_L,\Phi_L} (\chi^{(L)}_{i}) = \chi^{(L)}_{i},
  \]
since $ \chi^{(L)}_{i} \in W_L $ and the uniqueness of the interpolant.
 This leads to 
 \begin{align*}
  b_i &= \chi^{(L)}_{i} + \sum_{\ell = m}^{L-1}
  \sum_{\widetilde{\fraku} \subseteq \{\ell +1,\dots, L \}}
   (-1)^{| \widetilde{\fraku}|} \cI_{\widetilde{\fraku}}
   \left(I_{X_\ell,\Phi_\ell}(\chi^{(L)}_{i})\right) \\
  &= \chi^{(L)}_{i} + \sum_{\ell = m}^{L-1}
  \sum_{\widetilde{\fraku} \subseteq \{\ell+1,\dots, L \}}
   (-1)^{| \widetilde{\fraku}|} \cI_{\widetilde{\fraku}}
   \left(I_{X_\ell,\Phi_\ell}(\chi^{(\ell)}_{i}) \right),
 \end{align*}
where we used \cref{lem:InterpolationLagrangeFunc} to go from 
$ I_{X_\ell,\Phi_\ell}(\chi^{(L)}_{i}) $ to $ I_{X_\ell,\Phi_\ell}(
  \chi^{(\ell)}_{i}) $ in the 
  last step, which we were allowed to do as we have $i\le N_m\le N_\ell$.
  To arrive at the claim, we see that, for fixed $ m \leq \ell \leq L-1 $, 
\begin{align*}
\sum_{\widetilde{\fraku} \subseteq \{\ell +1,\dots, L \}}
   (-1)^{| \widetilde{\fraku}|} \cI_{\widetilde{\fraku}}\left(
   I_{X_\ell,\Phi_\ell}
  (\chi^{(\ell)}_{i})\right) 
  &=  I_{X_\ell,\Phi_\ell}
  (\chi^{(\ell)}_{i}) - A_{\{\ell+1, \dots, L\}} 
  \left(I_{X_\ell,\Phi_\ell}
  (\chi^{(\ell)}_{i}) \right) \\
  &= \chi^{(\ell)}_{i} - A_{\{\ell+1, \dots, L\}} 
  (\chi^{(\ell)}_{i}),
\end{align*}
using also (\ref{restrictedmultilevel}).
\end{proof}

As it turns out, we can even prove a relation of the 
image spaces $\widetilde{V}_\ell = A_\ell(C(\Omega))$ of multilevel
operators on different levels. 

\begin{corollary}\label{cor:NestedImageSpaces}
 For all $ 1 \leq \ell  \leq L $ we have
 $A_L(\widetilde{V}_\ell) = \widetilde{V}_\ell$. To be more precise,
 for any $f\in C(\Omega)$, we have $ A_L(A_\ell(f)) = A_\ell(f)$. This
 means in particular $A_L^2 = A_L$ and the nestedness of the the
 approximation spaces, i.e.
\[
 \widetilde{V}_1 \subseteq \widetilde{V}_2 \subseteq 
 \cdots \subseteq \widetilde{V}_L.
\]
\end{corollary}

\begin{proof}
Let $f\in C(\Omega)$ be given.   We prove $A_L(A_\ell(f)) = A_\ell(f)$
by induction on $m:=L-\ell\in\N_0$ for $0\le m\le L$. 
For $m=0$ we have $\ell=L$. The multilevel method interpolates on the
finest level, i.e. $(A_L f)|X_L=f|X_L$ and as it only depends on the
values at $X_L$ we thus have $A_L(A_L(f))=A_L(f)$.

Next, assume that the statement is true for any $m\le L-1$. To show it for
$m+1\le L$, we define $\ell$ via  $m+1=L-\ell$. The recursive definition of the multilevel
method and the induction hypothesis for $m=L-1-\ell$ yield
\begin{align*}
A_L(A_\ell(f)) &= A_{L-1}(A_\ell(f)) +
\cI_{X_L,\Phi_L}(A_\ell(f)-A_{L-1}(A_\ell(f)))\\
&= A_\ell(f) + \cI_{X_L,\Phi_L} (A_\ell(f)-A_\ell(f))\\
& = A_\ell(f),
\end{align*}
completing the induction. The statement $A_L^2=A_L$ immediately follows from
this. Moreover, we have $\widetilde{V}_\ell = A_\ell(C(\Omega))
\subseteq C(\Omega)$ and thus $\widetilde{V}_\ell =
A_{\ell+1}(\widetilde{V}_\ell) \subseteq A_{\ell+1}(C(\Omega)) =
\widetilde{V}_{\ell+1}$, proving the nestedness of the approximation spaces.
\end{proof}

This nestedness and the finite dimensionality of the approximation
spaces $\widetilde{V}_\ell$ allows us to decompose them in the form
\begin{equation}\label{decomp2}
\widetilde{V}_{\ell} = \widetilde{V}_{\ell-1} + \widetilde{W}_{\ell}
\end{equation}
and we will now prove that this sum is again a direct sum and
introduce another basis for $\widetilde{V}_L$, which also gives
another representation of the multilevel operator.

\begin{definition}
A Newton-type basis for the approximation space $\widetilde{V}_L$ is
defined by
 \[
 \widetilde{\cB} = \left\{ \chi_i^{(\ell)} : 1\le \ell \le L,
 N_{\ell-1}+1\le i\le N_\ell\right\}.
 \]
  \end{definition}
The next theorem will show that this is indeed a basis for
$\widetilde{V}_L$. It then follows that in the above decomposition
(\ref{decomp2}), the detail space $\widetilde{W}_\ell$ has the basis
$\{\chi_i^{(\ell)} : N_{\ell-1}+1 \le i\le N_{\ell}\}$. This new bases
demonstrates more accurately the inductive structure of the multilevel method. 
Again, we see that the additional basis elements 
$ \{\chi^{(\ell)}_{i}  :  N_{\ell-1} + 1 \leq i \leq N_\ell \} $
vanish on $ X_{\ell-1} $, showing a Newton-type behavior. 

\begin{theorem}\label{thrm:NewtonBasis}
  With $N_0:=0$ and  $A_0(f):= 0$, the multilevel approximation at level $ L $
 to $ f \in C(\Omega) $ can be written as 
\begin{equation}\label{rep2}
   A_L(f) = \sum_{\ell=1}^{L} \sum_{i=N_{\ell-1}+1}^{N_\ell}
   \left[f(\xx_i)-A_{\ell-1} f(\xx_i)\right] \chi_i^{(\ell)}.
\end{equation}
Consequently, the approximation space $\widetilde{V}_L=A_L(C(\Omega))$
is recursively given by
\[
\widetilde{V}_{L} =
\begin{cases}
  \operatorname{span}\left\{\chi_i^{(1)} : 1\le i\le N_1\right \} & \mbox{ if
  } L=1, \\
  \widetilde{V}_{L-1} \oplus \operatorname{span} \left\{\chi_i^{(L)} :
  N_{L-1}+1 \le i\le N_L\right\} & \mbox{ if } L\ge 2.
  \end{cases}
\]
\end{theorem}

\begin{proof}
 We prove the representation (\ref{rep2})  again by induction on $L$. For 
 $L =1 $ we have $A_1(f) = I_{X_1,\Phi_1}f$ and hence the statement. 
 
 For $ L \geq 2 $ we again use the definition of the multilevel
 operator to conclude
 \begin{align}
   A_L(f) &= A_{L-1}(f) + I_{X_L, \Phi_L}(f - A_{L-1}(f)) \nonumber\\
   &= A_{L-1}(f) + \sum_{i=1}^{N_L} \left( f(\xx_i) - 
   A_{L-1}(f)(\xx_i) \right) \chi^{(L)}_{i}\nonumber\\
   &= A_{L-1}(f) + \sum_{i=N_{L-1}+1}^{N_L} \left(f(\xx_i)- 
   A_{L-1}(f)(\xx_i) \right) \chi^{(L)}_{i},\label{eq:RecursiveRepresentationAj}
 \end{align}
 where we used that $A_{L-1}f|X_{L-1}= f|X_{L-1}$ in the last
 step. The induction hypothesis now leads to the representation
 (\ref{rep2}). Moreover, this shows that $\widetilde{\cB}$
spans $\widetilde{V}_L$ and a comparison of the dimensions shows that
$\widetilde{\cB}$ is even a basis of $\widetilde{V}_L$. This also
means that the stated sum of subspaces is direct.
\end{proof}

Additionally, \eqref{eq:RecursiveRepresentationAj} allows us to derive
a recursive representation of $ b_i $, where the recursion is done
over the levels. To emphasize the level-dependence,
we will write $ b_i^{(\ell)} $ for the $i$-th basis function of the
multilevel cardinal basis $\cB_\ell := \{b_i^{(\ell)} : 1\le i\le
N_\ell\}$ of  $ \widetilde{V}_{\ell} $. 

\begin{theorem}\label{thrm:RecursiveRepresentationBi}
The basis functions $ b_i^{(L)}\in\cB_L $ of the multilevel cardinal basis
satisfy 
\[
  b_i^{(L)} 
  = \begin{cases}
    \displaystyle b_i^{(L-1)} - \sum_{j = N_{L-1}+1}^{N_L} b_i^{(L-1)}(\xx_j)
    \chi^{(L)}_j & \mbox{for }   1 \leq i \leq N_{L-1},\\
    \chi_i^{(L)}&  \mbox{for }   N_{L-1}+1 \leq i \leq N_L.
  \end{cases}
  \]
\end{theorem}

\begin{proof}
Let $f \in C(\Omega)$ be fixed. We use the representation
\eqref{eq:RecursiveRepresentationAj} of $A_L(f)$ to derive
\begin{align*}
  A_L(f) &= A_{L-1}(f) + \sum_{j=N_{L-1}+1}^{N_L} \left(f(\xx_j)- 
  A_{L-1}(f)(\xx_j) \right) \chi^{(L)}_{j} \\
  &= \sum_{i=1}^{N_{L-1}} f(\xx_i) b_i^{(L-1)} 
  + \sum_{j=N_{L-1}+1}^{N_L} \left( f(\xx_j) - \sum_{i=1}^{N_{L-1}}
  f(\xx_i) b_i^{(L-1)}(\xx_j) \right) \chi_j^{(L)} \\ 
  &= \sum_{i=1}^{N_{L-1}} f(\xx_i) \left( b_i^{(L-1)} -
  \sum_{j=N_{L-1}+1}^{N_L} b_i^{(L-1)}(\xx_j) \chi_j^{(L)} \right) +
   \sum_{i=N_{L-1}+1}^{N_L} f(\xx_i) \chi_i^{(L)}.
 \end{align*}
A comparison to $ A_L(f) = \sum_{i=1}^{N_L} f(\xx_i) b_i^{(L)} $ finishes the proof.
\end{proof}

\subsection{Exponential Decay of the Bases}\label{subsec:ExponentialDecay}

It is well-established that the cardinal functions $ \chi_i^{(\ell)} $
of $W_\ell$ decay exponentially with growing $ \| \xx - \xx_i
\|_2/q_\ell $. A precise formulation  for nested data sets is in the
next theorem, though the nestedness is not really necessary. For its
proof we refer to  \cite[Theorem 2.3]{DeMarchi-Wendland-20-1}. 

\begin{theorem}\label{thrm:ExponentialDecayLagrangeFunctions}
 Let $ \Phi $ be a compactly supported reproducing kernel of $ H^{\sigma}(\R^d) $,
i.e., its Fourier transform satisfies \eqref{eq:AlgebraicDecay}.  Let
$ X_1\subseteq X_2\subseteq \cdots \subseteq X_L \subseteq \R^d $ be
a family of quasi-uniform sets 
 of data sites with fill distances $ h_1, \ldots, h_L $. For $ 1 \leq \ell \leq
 L $, let $ \Phi_{\ell} = \Phi_{\delta_\ell} $ be defined by
\eqref{eq:RescaledKernel} with $ \delta_\ell = \nu h_\ell $, with $ \nu
\geq 1 $.
  Then, there are constants $ C_1 > 0 $ and $ \eta > 0 $ that
 are independent of $ \ell $, such that the bound  
\[
  | \chi_i^{(\ell)}(\xx) | \leq C_1 e^{- \eta \frac{\| \xx - \xx_i
      \|_2}{q_\ell}}, \qquad \xx\in\R^d
\]
holds for all  $ 1 \leq \ell \leq L $ and $ 1 \leq i
\leq N_\ell $.

In particular, the basis functions of the Newton-type basis $\cB_{\operatorname{full}}$ of $V_L$ and the basis
functions from the Newton-type basis $\widetilde{\cB}$ of
$\widetilde{V}_{L}$ enjoy this type of exponential decay.
\end{theorem}

Following \eqref{eq:NewRepresentationBi}, the multilevel cardinal basis
$ \cB=\{b^{(L)}_i : 1\le i\le N_L\} $ of $\widetilde{V}_L$ is closely related to
the Newton-type basis $\widetilde{\cB}$. In particular, for  $ N_{L-1}+1 \leq i \leq N_L $,  the basis
function $b_i^{(L)} $ is even equal to $ \chi^{(L)}_i$.  Hence,  we expect
to obtain  a similar exponential decay  for the basis functions from $\cB$.

\begin{theorem} \label{thm:decay2} With the notation and assumptions of
 \cref{thrm:ExponentialDecayLagrangeFunctions}, fix $ L \in \N
 $. Then, there is a constant $C=C_L>0$ such that for all $1\le m\le
 L$ and $N_{m-1}+1\le i\le N_m$ the estimate
 \[
| b^{(L)}_i(\xx)| \le C_L e^{-\frac{\eta}{q_m} \| \xx -
  \xx_i \|_2}, \qquad \xx\in\R^d,
\]
holds, where $\eta>0$ is the constant from
\cref{thrm:ExponentialDecayLagrangeFunctions}.
\end{theorem}

\begin{proof}
  The proof is once again by induction on the level $L$. For $L=1$
  we note that we must have $m=1$ and $1\le i\le N_1$. Thus, we have
  $b_i^{(1)} = \chi_i^{(1)}$ and thus
  \cref{thrm:ExponentialDecayLagrangeFunctions} shows  the statement.

  For the inductions step, we assume that the result is correct for
  $L-1$ and thus   $1\le m\le L-1$. For level $L$ we might have $m=L$,
  which then again leads to $b_{i}^{(L)} = \chi_i^{(L)}$ for $N_{L-1}+1
  \le i\le N_L$ and Theorem \ref{thrm:ExponentialDecayLagrangeFunctions}  
  gives the desired statement. If $m\le L-1$, we can use the recursion
  \begin{equation}\label{bi1}
  b_i^{(L)}(\xx) = b_i^{L-1}(\xx) - \sum_{j=N_{L-1}+1}^{N_L}
    b_i^{(L-1)}(\xx_i) \chi_j^{(L)}(\xx)
    \end{equation}
    together with the induction assumption for bounding  $b_i^{(L-1)}$
    and \cref{thrm:ExponentialDecayLagrangeFunctions} for bounding
    $\chi_{j}^{(L)}$.
    This yields first of all
    \begin{align*}
      \left| b_i^{(L-1)}(\xx_i) \chi_j^{(L)}(\xx)\right|  & \le
      C_1C_{L-1}  e^{-\frac{\eta}{q_m}\|\xx_i-\xx_j\|_2}
      e^{-\frac{\eta}{q_L}\|\xx-\xx_j\|_2}\\
      &= C_1C_{L-1}
      e^{-\frac{\eta}{q_m}\left(\|\xx_i-\xx_j\|_2+\|\xx-\xx_j\|_2\right)}
      e^{-\|\xx-\xx_j\|_2\left(\frac{\eta}{q_L}-\frac{\eta}{q_m}
        \right)}\\
      &\le C_1C_{L-1}e^{-\frac{\eta}{q_m}\|\xx-\xx_i\|_2}
      e^{-\|\xx-\xx_j\|_2\left(\frac{\eta}{q_L}-\frac{\eta}{q_m} \right)}.
      \end{align*}
    Next, we turn to the sum over these terms. The above estimates
    immediately lead to
    \[
      \sum_{j=N_{L-1}+1}^{N_L}
      \left| b_i^{(L-1)}(\xx_i) \chi_j^{(L)}(\xx)\right| \le
      C_1C_{L-1} e^{-\frac{\eta}{q_m}\|\xx-\xx_i\|_2}
      \sum_{j=N_{L-1}+1}^{N_L}
      e^{-\|\xx-\xx_j\|_2\left(\frac{\eta}{q_L}-\frac{\eta}{q_m}
        \right)}
      \]
      and we now need to show that the latter sum can be bounded by a constant.
    To this end, we
    introduce $E_n = \{\yy\in\R^d : nq_L \le \|\xx-\yy\|_2<(n+1)
    q_L\}$ for $n\in \N_0$ and note that any $\xx_i\in X_L$ must be
    contained in exactly one of these $E_n$. Moreover, a simple
    volume argument  shows $|X_L\cap E_0|\le 2^d$ and, as in the proof of \cite[Theorem
      12.3]{Wendland-05-1},  $|X_L\cap E_n|\le 3^d  n^{d-1}$ for $n\ge
    1$. Thus, we find
    \begin{align*}
     \sum_{j=N_{L-1}+1}^{N_L}
     e^{-\|\xx-\xx_j\|_2\left(\frac{\eta}{q_L}-\frac{\eta}{q_m}\right)}
      & \le \sum_{n=0}^\infty \sum_{\xx_j\in E_n} e^{-nq_L
       \left(\frac{\eta}{q_L}-\frac{\eta}{q_m}\right)}\\
     & \le 2^d +  3^d \sum_{n=1}^\infty n^{d-1} e^{-n\eta\left(1 -
         \frac{q_L}{q_m}\right)}\\
       &\le 2^d +3^d \sum_{n=1}^\infty n^{d-1} e^{-n\eta\left(1 -
         \frac{q_L}{q_{L-1}}\right)} =: \widetilde{C}_L.
    \end{align*}
    Note that our assumptions on $q_\ell$ and $h_\ell$ give the bound
    $q_L/q_{L-1} \le \mu c_{qu}$. Thus, if $\mu c_{qu}<1$ then we can
    choose the constant $\widetilde{C}_L$ even independently of $L$. 

    In any case, plugging this all into \eqref{bi1}, and using the induction hypothesis
on the first term again yields
    \begin{align*}
 |b_i^{(L)}(\xx)| &\le  \left| b_i^{L-1}(\xx)\right|  + \sum_{j=N_{L-1}+1}^{N_L}
 \left| b_i^{(L-1)}(\xx_i) \chi_j^{(L)}(\xx)\right| \\
 &\le C_{L-1} e^{-\frac{\eta}{q_m}\|\xx-\xx_i\|_2} +
 C_1C_{L-1}  \widetilde{C}_L 
 e^{-\frac{\eta}{q_m}\|\xx-\xx_i\|_2}
 = C_L  e^{-\frac{\eta}{q_m}\|\xx-\xx_i\|_2}
    \end{align*}
    with $C_L:=C_{L-1}(1+C_1 \widetilde{C}_L)$.
  \end{proof}

While the basis functions of both $\cB$ and $\widetilde{\cB}$ are now
proven to have an exponential decay there is an important difference
between them. For the basis functions $\chi_i^{(\ell)}$ of
$\widetilde{\cB}$ the constant $C_1>0$ is independent of the level,
while the constant $C_L>0$ for the basis functions $b_i^{(L)}$ from
$\cB$ depends on the highest level $L$. As this does not show up in
numerical calculations, it remains an open problem whether this
constant can also be chosen independently of $L$.

An exponential decay of the basis functions immediately leads to
uniform boundedness of their $\ell_1$-sum, as shown in the next corollary.

\begin{corollary}
The $\ell_1$-norm of the basis functions in $\cB$ and
in $\widetilde{\cB}$ are uniformly bounded. To be more precise there
exist a level independent constant $\widetilde{C}_1>0$ and a level
dependent constant $\widetilde{C}_L>0$ such that
\[
\sum_{j=1}^N |b_j^{(L)}(\xx)| \le \widetilde{C}_L, \qquad
\sum_{\ell=1}^L
\sum_{i=N_{\ell-1}+1}^{N_\ell}|\chi_i^{(\ell)}(p\xx)|\le
\widetilde{C}_1
\]
for all $\xx\in \R^d$.
\end{corollary}

\begin{proof}
As the proof is the same for both bases, we let $\phi_i$ be either
$\chi_i^{(\ell)}$ or $b_i^{(L)}$ for some $N_{\ell-1}+1\le i\le N_\ell$
and $1\le \ell\le L$. Then, the exponential decay gives
\[
|\phi_i(\xx)| \le C e^{-\frac{\eta}{q_\ell}\|\xx-\xx_i\|_2} \le C
e^{-\frac{\eta}{q_1}\|\xx-\xx_i\|_2} , \qquad \xx\in\R^d.
\]
This time, letting $E_n = \{\yy\in\R^d : nq_1 \le \|\xx-\yy\|_2< (n+1)q_1\}$
yields, as in the proof of Theorem \ref{thm:decay2},
\[
  \sum_{i=1}^{N_L} |\phi_i(\xx)| \le C \sum_{n=0}^\infty \sum_{\xx_i\in
    E_n} e^{-\eta n} \le C\left(2^d + \sum_{n=0}^\infty n^{d-1}e^{-\eta
      n}\right),
    \]
where the sum is clearly a level-independent constant. Hence, the overall
constant is only level-dependent for the multilevel cardinal basis but
not the Newton-type basis.    
\end{proof}

Note that the result on $\cB$ means that the associated Lebesgue function, which
is simply the $\ell_1$-norm of the basis functions in the case of a
cardinal basis and which coincides with $\|A_L\|_{L_\infty\to
  L_\infty}$ is bounded by $\widetilde{C}_L$. Unfortunately, again the
question arises whether this can be done with a constant that is
independent of the level. Only in that case, it immediately follows that
$A_L(f)$ also converges to $f$ point-wise for even any function $f\in
C(\overline{\Omega})$. 

Similarly, it is possible to show that the basis functions of both bases
are Lipschitz continuous. Following \cite{Hangelbroek-etal-11-1}, this then
leads to  the first step in the direction of showing that $\cB$ is a stable
basis. To be more precise,  for an $f = \sum_{j=1}^{N_L}
a_jb_j^{(L)}$ let $P_\ell f = \sum_{j={N_{\ell-1}+1}}^{N_\ell}a_jb_j^{(L)}$ and
$\aa_\ell=(a_{N_{\ell-1}+1},\ldots, a_{N_\ell})^\transpose$ then the
general theory of \cite{Hangelbroek-etal-11-1} shows that
\[
c_1q_\ell^{d/p} \|\aa_\ell\|_{\ell_p} \le \|P_\ell f\|_{L_p(\Omega)}\le
c_2q_\ell^{d/p} \|\aa_\ell\|_{\ell_p}.
\]
Unfortunately, here the constants $c_1,c_2>0$ depend so far also on the
level $L$. Thus, we leave out the details.

\subsection{Remarks on the Numerical Computation}\label{subsec:NumericalComputations}
From a numerical point of view, computing the cardinal bases
$\chi_{i}^{(\ell)}$ should be avoided if possible. However, using the
multilevel cardinal basis has advantages when it comes to tensorizing
such bases for higher dimensional problems. Hence, we shortly describe
how this can be done and analyze the computational cost. To this end,
we start for a fixed index $i$ with a representation of the from
\[
 b_i = A_L(\ee^{(L)}_i) =  A_L(\chi_i^{(L)}) = \sum_{\ell=1}^{L} \sum_{j = 1}^{N_{\ell}}
 \alpha_j^{(\ell)} \Phi_\ell( \cdot - \xx_j).
\]
The coefficient vectors $ \aalpha^{(m)} $, 
$1 \leq \ell \leq L $, are the unique solutions of the 
block linear system 
\[ 
 \begin{pmatrix}
      M_1 \\
    B_{21} & M_{2}    &        &            & \\
    B_{31} & B_{32} & M_{3}    &            & \\
    \vdots & \vdots & \cdots & \ddots     & \\
    B_{L1} & B_{L2} & \cdots & B_{L(L-1)} & M_L 
   \end{pmatrix}
   \begin{pmatrix}
    \aalpha^{(1)} \\
    \aalpha^{(2)} \\
    \aalpha^{(3)} \\
    \vdots        \\
    \aalpha^{(L)}
   \end{pmatrix}
   = \begin{pmatrix}
      \chi^{(L)}_i|_{X_1} \\
      \chi^{(L)}_i|_{X_2} \\
      \chi^{(L)}_i|_{X_3} \\
      \vdots \\
      \chi^{(L)}_i|_{X_L}
 \end{pmatrix},
\]
where $ M_\ell = ( \Phi_\ell(\xx_{i} - \xx_{j}))_{1 \leq i,j \leq
  N_\ell} $ is the interpolation matrix on level $ \ell $, and
$B_{\ell m} =  
(\Phi_m(\xx_{i} - \xx_{j}))_{\substack{1 \leq i \leq N_{\ell} \\ 1
    \leq j \leq N_m}} $ is the transportation matrix from level $ m $
to level $ \ell > m $. 

For a fixed level $L$, the nestedness of the data sites 
enforces a clear recursive structure. If the anchor $ \xx_i $ of $ b_i
$ appears on the finest level, i.e. $ N_{L-1} +1 \leq i \leq N_L $,
then we only have to solve the subsystem  
$
 M_L \aalpha^{(L)} = \ee_i. 
$
If $ \xx_i $ appears on the second to last level, i.e., $ N_{L-2} + 1
\leq i \leq N_{L-1} $ we have to solve the subsystems on levels  
$L-1$ and $L$. In general, we only have to solve the whole block linear system only for those $ \xx_i $ that appear in the first level. 
Thus, the complexity at level $L$ is governed by this 
recursive coupling, as summarized in the following corollary.

\begin{corollary}
Assume that the sets $ \{X_\ell : 1\le \ell \le L\} $ are
quasi-uniform. Then, the computation of all necessary coefficients of
$\{b_i^{(L)} : 1 \leq i \leq N_L\} $  
can be done in $ \mathcal{O}(N_L^2 \log(N_L))$
time, if the systems can be solved in linear time. Keeping them in memory needs 
\[
 \mathcal{O} \left(\sum_{\ell=1}^{L} (N_\ell - N_{\ell-1})
 \sum_{m=\ell}^L N_m \right)
\]
space.
A single point evaluation of $ A_L(f) $ in nodal
 representation takes
$
  \mathcal{O} (N_L \log(N_L))
$
 time.
\end{corollary}

\begin{proof}
 Assuming that the matrices are build, we have to solve 
 $ N_L $ many multilevel problems. Each takes $ \cO \left(N_L \log(N_L)\right) $ time. 
 The second claim follows directly from the ideas above.
 
 To see the cost of a single point-evaluation, we write $ A_L(f)(\xx) $ as 
\[
  A_L(f)(\xx) = \sum_{i=1}^{N_L} f(\xx_{i}))b_i^{(L)}(\xx) 
  = \sum_{i=1}^{N_L} f(\xx_i) \sum_{\ell=1}^{L} \sum_{j = 1}^{N_{\ell}}
 \alpha_j^{(\ell)} \Phi_\ell( \xx - \xx_j).
\]
The innermost summation is actually only over those indices $j$ with
$\|\xx-\xx_j\|_2\le \delta_\ell$. As the sets are quasi-uniform this can
be done in constant time. 
Together with  $ L \sim \log(N_L) $, we arrive at the claimed cost.
\end{proof}

\begin{remark}
 Both the computation of the coefficients and the point evaluation is
 by a factor of $ N_L $ more expensive than the
 representation of the multilevel operator in
 \Cref{def:MultilevelOperator}. Hence, the  multilevel cardinal basis $\cB$ should only be
 used when a representation with separated sampling and evaluation
 points is required. 
\end{remark}

\section{The Nodal Representation for Non-nested Sets}\label{sec:NonNestedSites}
We now shift our focus on non-nested point sets
$ \{X_\ell :1 \leq \ell \leq L\} $, i.e., we do not assume
$X_{\ell}\subseteq X_{\ell+1}$ but we also do not assume that
$X_\ell\cap X_{m} = \emptyset$ for $\ell\ne m$.

One way of handling this situation is to introduce a new sequence
$\{Y_\ell : 1\le \ell\le L\}$ of nested point sets  
\[
 Y_\ell:= X_1 \cup \cdots \cup X_\ell =: \{\yy_1,\ldots,\yy_{n_\ell}\},
\]
having $n_\ell:= |Y_\ell|$ points.

Then, we can proceed as in the nested case and compute a multilevel
approximation which we denote by $A_L(Y;\cdot)$. Obviously, all
previous results hold for this new sequence of data sets, i.e. we have
a multilevel cardinal basis $b_i:=A_L(Y;\ee_i^{(L)})$, which is a
cardinal basis with respect to $Y_L$, and, similarly, a
Newton-type basis, as well. However, these are not bases for the
approximation spaces $V_\ell$ built with the sequence $\{X_\ell\}$ but
rather for the spaces $V_\ell(Y)$ built with the sequence $\{Y_\ell\}$.

If we want to use the multilevel algorithm with the original sequence
of data sets, we can still derive a basis of the space
$\widetilde{V}_L$ and a representation of $A_L=A_L(X;\cdot)$. We still
have that $A_L$ is linear and uses only the data
$f|X_\ell$ for $1\le \ell\le L$. The latter means that it uses the
data $f|Y_L$. Thus, writing once again $f|Y_L = \sum_{i=1}^{n_L}
f(\yy_i)\ee_i$ with the $i$-th unit vector $\ee_i\in\R^{n_L}$ we see
\[
A_L(X;f) = \sum_{j=1}^{n_L} f(\yy_i) A_L(X;\ee_i).
\]
Thus, we again see that $\widetilde{V}_L = A_L(X;C(\Omega))$ is
spanned by $\cB(Y_L) = \{b_i=A_L(X;\ee_i) : 1\le i\le n_L\}$. We can not expect that 
the so-defined functions $ \{b_i\} $ are cardinal. Nevertheless, we
can still prove some properties of $\widetilde{V}_L$. 

\begin{theorem}
  In the case of  non-nested sets $\{X_\ell\}$ let
  $Y_L=X_1\cup\cdots\cup X_L$ and let $\cB(Y_L) = \{b_i=A_L(X;\ee_i) : 1\le i\le n_L\}$.  
 Then, the multilevel approximation $A_L(X;f)$ to $ f \in C(\Omega) $ is given by 
\[
A_L(X;f) = \sum_{i=1}^N f(\yy_i) b_i
\]
Moreover, $\cB(Y_L)$ forms a basis of $\widetilde{V}_L$, showing
$\dim \widetilde{V}_L=n_L$.
\end{theorem}

\begin{proof}
We already know that $A_L(X;f)$ can be written in the stated form and
that $\cB(Y_L)$ spans $\widetilde{V}_L$. Next, we will show that
$\cB(Y_L)$ is linearly independent.  To this end, assume that 
$
0 =  \sum_{i=1}^{n_L} \alpha_i b_i(\xx)$ for all $ \xx \in \Omega$.

Next, to simplify the notation, we choose an $f\in C(\Omega)$ with $f|Y_L=\aalpha$.
Then, returning to the original form of the multilevel method
with local approximations $s_\ell$ and errors $e_\ell$ we find
\[
0 =  A_L(X;f)(\xx) = \sum_{\ell=1}^L s_\ell(\xx), \qquad \xx\in\Omega.
\]
As the global approximation space is a direct sum of the local ones, $
V_L = W_1 \oplus \cdots \oplus W_L $, we immediately can conclude that
$s_\ell = 0$ for $1\le \ell \le L$, 
which implies $ e_\ell = f - (s_1 + \cdots + s_\ell ) = f $ for all $
1 \leq \ell \leq L $. This, in turn, leads to  
\[
 0 = s_\ell = I_{X_\ell,\Phi_\ell}(e_{\ell-1}) = I_{X_\ell,\Phi_\ell}(f),
\]
showing $ f|_{X_\ell} = 0 $ for all $ 1 \leq j \leq L $. Consequently,
this means $ \aalpha = \00 \in \R^{n_L} $ and hence the linear
independence. of $\cB(Y_L)$.

To see that $\widetilde{V}_L$ is not a
proper subspace of $\operatorname{span}(\cB(Y_L))$, we can proceed
exactly as in the proof of Theorem \ref{thm:basis1}.
\end{proof}

\begin{remark}
Non-nested point sets appear naturally in \emph{adaptive multilevel approximations},
where points at the next level are used if the current residual is larger than a given 
threshold, see \cite{Wendland-17-1}. This has obviously a connection to
a greedy selection of points, 
which is also a possible extension of the multilevel method introduced
in \Cref{sec:MultilevelMethod}. 
\end{remark}

\section{Conclusion and Outlook}\label{sec:Conclusion}
In this paper we studied the structure of multilevel kernel approximation spaces.
Our main contribution is a nodal representation of the multilevel
operator that makes its dependence on the data explicit. This
representation reveals the structure of the associated global
approximation space and allows the construction of basis functions
adapted to the multilevel setting. 

For nested families of data sites we derived both a Lagrange basis and
a Newton-type basis of the multilevel space. These bases provide a
direct description of the range of the multilevel operator and enable
a detailed analysis of structural properties such as localization and
computational complexity. In particular, we established exponential
decay of the Lagrange basis functions, which indicates that the
multilevel representation inherits favorable localization properties
from the underlying kernel. 

We also investigated the more general case of non-nested
sets. Although a cardinal basis is no longer available in this
setting, the ideas developed for the nested case still allow the
construction of a basis for the global approximation space and yield a
characterization of the range of the multilevel operator.  

The nodal representation derived in this work provides a practical
description of the multilevel operator that is suitable for
high-dimensional constructions, for instance when combined with sparse
tensor techniques such as Smolyak-type algorithms. 

Several directions for future research remain open. From a theoretical
perspective it would be desirable to obtain sharper localization and
stability estimates for the basis functions. Another interesting
direction is the analysis of 
adaptive strategies which dynamically generate non-nested data
sites. From a computational viewpoint, the explicit representation of
the multilevel operator may enable the development of efficient
algorithms for high-dimensional approximation based on sparse tensor
constructions and kernel methods. Finally, it would be interesting to
investigate extensions of the present framework to more general
approximation settings, including anisotropic kernels and operator
learning problems.

\bibliographystyle{siamplain}

\begin{thebibliography}{10}

\bibitem{Arcangeli-etal-12-1}
{\sc R.~Arcang\'{e}li, M.~C.~L. de~Silanes, and J.~J. Torrens}, {\em Extension
  of sampling inequalities to {S}obolev semi-norms of fraction order and
  derivative data}, Numer.\ Math., 121 (2012), pp.~587--608.

\bibitem{Avesani-etal-25-1}
{\sc S.~Avesani, R.~Kempf, M.~Multerer, and H.~Wendland}, {\em Multiscale
  scattered data analysis in samplet coordinates}, SIAM Journal on Scientific
  Computing, 47 (2025), pp.~A3038--A3063.

\bibitem{Buhmann-03-1}
{\sc M.~D. Buhmann}, {\em Radial Basis Functions}, Cambridge Monographs on
  Applied and Computational Mathematics, Cambridge University Press, Cambridge,
  2003.

\bibitem{Buettner-etal-25-1}
{\sc M.~B{\"u}ttner, R.~Kempf, and H.~Wendland}, {\em Numerical aspects of the
  tensor product multilevel method for high-dimensional, kernel-based
  reconstruction on sparse grids}, J. Sci. Comput., 106 (2026).

\bibitem{Chen-etal-02-1}
{\sc C.~S. Chen, M.~Ganesh, M.~A. Golberg, and A.~H.-D. Cheng}, {\em Multilevel
  compact radial functions based computational schemes for some elliptic
  problems}, Computers \& Mathematics with Applications, 43 (2002), pp.~359 --
  378.

\bibitem{Cristianini-Shawe-Taylor-00-1}
{\sc N.~Cristianini and J.~Shawe-Taylor}, {\em An introduction to support
  vector machines and other kernel-based learning methods}, Cambridge
  University Press, Cambridge, 2000.

\bibitem{Cucker-Zhou-07-1}
{\sc F.~Cucker and D.~X. Zhou}, {\em Learning Theory: An Approximation Theory
  Viewpoint}, Cambridge Monographs on Applied and Computational Mathematics,
  Cambridge University Press, Cambridge, 2007.

\bibitem{Fasshauer-07-1}
{\sc G.~Fasshauer}, {\em Meshfree Approximation Methods with {MATLAB}}, World
  Scientific Publishers, Singapore, 2007.

\bibitem{Fasshauer-99-1}
{\sc G.~E. Fasshauer}, {\em Solving differential equations with radial basis
  functions: Multilevel methods and smoothing}, Adv.\ Comput.\ Math., 11
  (1999), pp.~139--159.

\bibitem{Floater-Iske-96-1}
{\sc M.~S. Floater and A.~Iske}, {\em Multistep scattered data interpolation
  using compactly supported radial basis functions}, J.\ Comput.\ Appl.\ Math.,
  73 (1996), pp.~65--78.

\bibitem{Fornberg-Flyer-15-1}
{\sc B.~Fornberg and N.~Flyer}, {\em Solving {PDEs} with radial basis
  functions}, in Acta Numerica, A.~Iserles, ed., vol.~24, Cambridge University
  Press, 2015, pp.~215--258.

\bibitem{Franke-Schaback-98-1}
{\sc C.~Franke and R.~Schaback}, {\em Solving partial differential equations by
  collocation using radial basis functions}, Appl.\ Math.\ Comput., 93 (1998),
  pp.~73--82.

\bibitem{Georgoulis-etal-13-1}
{\sc E.~H. Georgoulis, J.~Levesley, and F.~Subhan}, {\em Multilevel sparse
  kernel-based interpolation}, SIAM J.\ Sci. Comput., 35 (2013),
  pp.~A815--A831.

\bibitem{Griebel-etal-26-1}
{\sc M.~Griebel, H.~Harbrecht, and M.~Multerer}, {\em Kernel interpolation on
  generalized sparse grids}, SIAM Journal on Mathematics of Data Science, 8
  (2026), pp.~335--361.

\bibitem{Hales-Levesley-02-1}
{\sc S.~J. Hales and J.~Levesley}, {\em Error estimates for multilevel
  approximation using polyharmonic splines}, Numer.\ Algorithms, 30 (2002),
  pp.~1--10.

\bibitem{Hangelbroek-etal-11-1}
{\sc T.~Hangelbroek, F.~J. Narcowich, X.~Sun, and J.~Ward}, {\em Kernel
  approximation on manifolds {II}: the {$L_\infty$} norm of the {$L_2$}
  projector}, SIAM J.\ Math.\ Anal., 43 (2011), pp.~662--684.

\bibitem{Kempf-Wendland-23-1}
{\sc R.~Kempf and H.~Wendland}, {\em High-dimensional approximation with
  kernel-based multilevel methods on sparse grids}, Numerische Mathematik, 154
  (2023), pp.~485--519.

\bibitem{LeGia-etal-10-1}
{\sc Q.~T. Le~Gia, I.~Sloan, and H.~Wendland}, {\em Multiscale analysis in
  {S}obolev spaces on the sphere}, SIAM J.\ Numer.\ Anal., 48 (2010),
  pp.~2065--2090.

\bibitem{LeGia-etal-12-1}
{\sc Q.~T. Le~Gia, I.~Sloan, and H.~Wendland}, {\em Multiscale {RBF}
  collocation for solving {PDEs} on spheres}, Numer.\ Math., 121 (2012),
  pp.~99--125.

\bibitem{DeMarchi-Wendland-20-1}
{\sc S.~D. Marchi and H.~Wendland}, {\em On the convergence of the rescaled
  localized radial basis function method}, Appl. Math. Lett., 99 (2020).

\bibitem{Narcowich-etal-99-1}
{\sc F.~J. Narcowich, R.~Schaback, and J.~D. Ward}, {\em Multilevel
  interpolation and approximation}, Appl.\ Comput.\ Harmon.\ Anal., 7 (1999),
  pp.~243--261.

\bibitem{Narcowich-etal-05-1}
{\sc F.~J. Narcowich, J.~D. Ward, and H.~Wendland}, {\em {S}obolev bounds on
  functions with scattered zeros, with applications to radial basis function
  surface fitting}, Math.\ Comput., 74 (2005), pp.~643--763.

\bibitem{Narcowich-etal-06-1}
{\sc F.~J. Narcowich, J.~D. Ward, and H.~Wendland}, {\em {S}obolev error
  estimates and a {B}ernstein inequality for scattered data interpolation via
  radial basis functions}, Constr.\ Approx., 24 (2006), pp.~175--186.

\bibitem{Schaback-95-2}
{\sc R.~Schaback}, {\em Error estimates and condition number for radial basis
  function interpolation}, Adv.\ Comput.\ Math., 3 (1995), pp.~251--264.

\bibitem{Schaback-97-3}
{\sc R.~Schaback}, {\em On the efficiency of interpolation by radial basis
  functions}, in Surface Fitting and Multiresolution Methods, A.~L.
  M\'ehaut\'e, C.~Rabut, and L.~L. Schumaker, eds., Nashville, 1997, Vanderbilt
  University Press, pp.~309--318.

\bibitem{Schaback-Wendland-06-1}
{\sc R.~Schaback and H.~Wendland}, {\em Kernel techniques: From machine
  learning to meshless methods}, in Acta Numerica, A.~Iserles, ed., vol.~15,
  Cambridge University Press, 2006, pp.~543--639.

\bibitem{Schoelkopf-Smola-02-1}
{\sc B.~Sch\"olkopf and A.~J. Smola}, {\em Learning with Kernels -- Support
  Vector Machines, Regularization, Optimization, and Beyond}, MIT Press,
  Cambridge, Massachusetts, 2002.

\bibitem{Stein-71-1}
{\sc E.~M. Stein}, {\em Singular Integrals and Differentiability Properties of
  Functions}, Princeton University Press, Princeton, New Jersey, 1971.

\bibitem{Steinwart-Christmann-08-1}
{\sc I.~Steinwart and A.~Christmann}, {\em Support Vector Machines}, Springer,
  New York, 2008.

\bibitem{Usta-Levesley-18-1}
{\sc F.~Usta and J.~Levesley}, {\em Multilevel quasi-interpolation on a sparse
  grid with the {G}aussian}, Numer.\ Algorithms, 77 (2018), pp.~793--808.

\bibitem{Wendland-99-2}
{\sc H.~Wendland}, {\em Meshless {G}alerkin methods using radial basis
  functions}, Math.\ Comput., 68 (1999), pp.~1521--1531.

\bibitem{Wendland-05-1}
{\sc H.~Wendland}, {\em Scattered Data Approximation}, Cambridge Monographs on
  Applied and Computational Mathematics, Cambridge University Press, Cambridge,
  UK, 2005.

\bibitem{Wendland-10-1}
{\sc H.~Wendland}, {\em Multiscale analysis in {S}obolev spaces on bounded
  domains}, Numer.\ Math., 116 (2010), pp.~493--517.

\bibitem{Wendland-17-1}
{\sc H.~Wendland}, {\em Multiscale radial basis functions}, in Frames and Other
  Bases in Abstract and Function Spaces -- Novel Methods in Harmonic Analysis,
  Volume 1, I.~Pesenson, Q.~T.~L. Gia, A.~Mayeli, H.~Mhaskar, and D.-X. Zhou,
  eds., Birkh\"auser, Cham, 2017, pp.~265--299.

\bibitem{Wendland-18-1}
{\sc H.~Wendland}, {\em Solving partial differential equations with multiscale
  radial basis functions}, in Contemporary Computational Mathematics - A
  Celebration of the 80th Birthday of Ian Sloan, J.~Dick, F.~Kuo, and
  H.~Wozniakowski, eds., Springer, Cham, 2018, pp.~1191--1213.

\end{thebibliography}

\end{document}